\documentclass[12pt,a4paper]{article}
\usepackage{latexsym,amssymb,amsfonts,amsmath,amsthm,nccmath,float}
\usepackage[hidelinks]{hyperref}
\usepackage[margin=2cm]{geometry}

\newtheorem{theorem}{Theorem}

\newtheorem{lemma}[theorem]{Lemma}
\newtheorem*{repCov}{Theorem \ref{NeatCoveringTheorem}}
\newtheorem*{repPac}{Theorem \ref{NeatPackingTheorem}}

\theoremstyle{definition}

\def \deg {{\rm deg}}
\def \wt {{\rm wt}}
\def \rank {{\rm rank}}

\def \diag {{\rm diag}}

\def \leq {\leqslant}
\def \geq {\geqslant}

\def \mod#1{{\:({\rm mod}\ #1)}}

\let\oldproofname=\proofname
\renewcommand{\proofname}{\rm\bf{\oldproofname}}

\title{\bf Generalising Fisher's inequality to coverings and packings}

\author{
Daniel Horsley\\
School of Mathematical Sciences \\
Monash University \\
Vic 3800, Australia \\[0.1cm]
\texttt{danhorsley@gmail.com}
}

\date{}

\begin{document}
\sloppy
\maketitle
\def\baselinestretch{1.2}\small\normalsize

\begin{abstract}
In 1940 Fisher famously showed that if there exists a non-trivial $(v,k,\lambda)$-design then $\lambda(v-1) \geq k(k-1)$. Subsequently Bose gave an elegant alternative proof of Fisher's
result. Here, we show that the idea behind Bose's proof can be generalised to obtain new bounds on the number of blocks in $(v,k,\lambda)$-coverings and -packings with $\lambda(v-1)<k(k-1)$.
\end{abstract}

\section{Introduction}

Let $v$, $k$ and $\lambda$ be positive integers and let $(V,\mathcal{B})$ be a pair where $V$ is a $v$-set of \emph{points} and $\mathcal{B}$ is a collection of $k$-subsets of $V$, called
\emph{blocks}. If each pair of points occur together in at least $\lambda$ blocks, then
$(V,\mathcal{B})$ is a \emph{$(v,k,\lambda)$-covering}. If each pair of points occur together in at most $\lambda$ blocks, then $(V,\mathcal{B})$ is a \emph{$(v,k,\lambda)$-packing}. If each pair of
points occur together in exactly $\lambda$ blocks, then $(V,\mathcal{B})$ is a
\emph{$(v,k,\lambda)$-design}. We refer to parameter sets $(v,k,\lambda)$ that do not satisfy $3 \leq k < v$, and designs with such parameter sets, as \emph{trivial}.

Usually we are interested in finding coverings with as few blocks as possible and packings with as
many blocks as possible. The \emph{covering number} $C_{\lambda}(v,k)$ is the minimum number of
blocks in any $(v,k,\lambda)$-covering and the \emph{packing number} $D_{\lambda}(v,k)$ is the
maximum number of blocks in any $(v,k,\lambda)$-packing. When $\lambda=1$ we omit the subscripts. For a given $k$ and $\lambda$, it is obvious that $D_{\lambda}(v,k) \leq D_{\lambda}(v',k)$ and $C_{\lambda}(v,k) \leq C_{\lambda}(v',k)$ when $v \leq v'$. The classical bound for covering numbers is the \emph{Sch{\"o}nheim bound} \cite{Sch} which states that
$$C_{\lambda}(v,k) \geq
    \left\lceil \mfrac{vr}{k} \right\rceil
    \quad \mbox{where} \quad r=\left\lceil\mfrac{\lambda(v-1)}{k-1}\right\rceil.$$
The classical bound for
packing numbers is the \emph{Johnson bound} \cite{Joh} which states that
$$D_{\lambda}(v,k) \leq \left\lfloor \mfrac{vr}{k} \right\rfloor
    \quad \mbox{where} \quad r=\left\lfloor\mfrac{\lambda(v-1)}{k-1}\right\rfloor.$$
These bounds are easily proved by observing that each point in a $(v,k,\lambda)$-covering occurs in at least $\lceil \frac{\lambda(v-1)}{k-1}\rceil$ blocks and each point in a $(v,k,\lambda)$-packing appears in at most $\lfloor\frac{\lambda(v-1)}{k-1}\rfloor$ blocks. A simple proof allows each of
these bounds to be improved by 1 in the case where $\lambda(v-1) \equiv 0 \mod{k-1}$ and $\lambda
v(v-1) \equiv 1 \mod{k}$ (see \cite{MilMul}, for example). Keevash's recent breakthrough result \cite[Theorem 6.5]{Kee} implies that, for a fixed $k$ and $\lambda$, $C_{\lambda}(v,k)$ and $D_{\lambda}(v,k)$ equal the improved Sch{\"o}nheim and Johnson bounds for all sufficiently large $v$. This represents the culmination of a large amount of work on the asymptotic behaviour of covering and packing numbers (see, for example, \cite{CarYus97,CarYus98,CheColLinWil,ErdHan,ErdRen,Rod}).

For packings with $\lambda=1$ we also have the \emph{second Johnson bound} \cite{Joh} which states that $D(v,k)(D(v,k)-1) \geq x(x-1)v+2xy$ where $x$ and $y$ are the integers such that $D(v,k)=xv+y$ and $0 \leq y <v$. This implies the slightly weaker statement that
$$D(v,k) \leq \left\lfloor\mfrac{v(k-1)}{k^2-v}\right\rfloor.$$

A number of results have been proved which improve on the Sch{\"o}nheim bound in various cases in which $k$ is a significant fraction of $v$ \cite{BluGreHei,BosCon,BryBucHorMaeSch,Fur,TodLB,TodFP}. Exact covering and packing numbers are known for $k \in \{3,4\}$. Also, exact covering numbers have been determined when $\lambda=1$ and $v \leq \frac{13}{4}k$ \cite{GreLiVan,Mil}. For surveys on coverings and packings see \cite{GorSti,MilMul,StiWeiYin}. Gordon maintains a repository for small coverings \cite{Gor}.

One of the most fundamental results in the study of block designs is \emph{Fisher's inequality} \cite{Fis} which states that any non-trivial $(v,k,\lambda)$-design has at
least $v$ blocks (or, equivalently that if there exists a non-trivial $(v,k,\lambda)$-design, then $\lambda(v-1) \geq k(k-1)$). Designs with exactly $v$ blocks (equivalently, those with $\lambda(v-1) = k(k-1)$) are called \emph{symmetric designs}. Many families of symmetric designs are known to exist, the most famous example being projective planes.

In \cite{Bos}, Bose gave an elegant alternative proof of Fisher's inequality. In this paper we show that the idea behind Bose's proof can be generalised to obtain new bounds on covering and packing numbers for parameter sets with $\lambda(v-1) < k(k-1)$. The most easily stated of our results are as follows.

\begin{theorem}\label{NeatCoveringTheorem}
Let $v$, $k$ and $\lambda$ be positive integers such that $3 \leq k <v$, and let $r$ and $d$ be the integers such that $\lambda(v-1)=r(k-1)-d$ and $0 \leq d < k-1$. If $d < r-\lambda$, then
$$C_{\lambda}(v,k) \geq \left\lceil\mfrac{v(r+1)}{k+1}\right\rceil.$$
\end{theorem}

\begin{theorem}\label{NeatPackingTheorem}
Let $v$, $k$ and $\lambda$ be positive integers such that $3 \leq k <v$, and let $r$ and $d$ be the integers such that $\lambda(v-1)=r(k-1)+d$ and $0 \leq d < k-1$. If $d < r-\lambda$, then
$$D_{\lambda}(v,k) \leq \left\lfloor\mfrac{v(r-1)}{k-1}\right\rfloor.$$
\end{theorem}

When the hypotheses of these theorems are satisfied, the bounds
they give are at least as good as the Sch{\"o}nheim bound and the first Johnson bound if $r<k$ and never improve on them otherwise. It can be seen that each of these theorems implies Fisher's inequality by observing that, if there exists a $(v,k,\lambda)$-design, then $C_{\lambda}(v,k)=D_{\lambda}(v,k)=\frac{vr}{k}$ and $r=\frac{\lambda(v-1)}{k-1}$. Theorem
\ref{NeatCoveringTheorem} also subsumes various results from \cite{BosCon} and
\cite{BryBucHorMaeSch}. In the discussion following its proof we show that, when $k$ is large in comparison with $i$ and $\lambda$, the bound of Theorem \ref{NeatCoveringTheorem} exceeds the Sch{\"o}nheim bound by $i$ or more for almost half of the possible parameter sets for which $r<k$. In contrast, previous results yield improvements for only an insignificant fraction of the possible parameter sets for which $r<k$.

Theorem \ref{NeatPackingTheorem} and the other theorems concerning packings in this paper are only of interest for $\lambda \geq 2$, because they are
invariably weaker than the second Johnson bound in the case $\lambda=1$. Because of this, and in
order to avoid repetition, we concentrate on the case of coverings when discussing our results.

In Section \ref{prelimSec} we introduce the notation and preliminary results that we require, and
in Section \ref{basicSec} we prove and discuss Theorems \ref{NeatCoveringTheorem} and \ref{NeatPackingTheorem}. In Section \ref{indepSec} we prove some results concerning $m$-independent sets in edge-weighted graphs, and then in Sections \ref{dBigSec} and \ref{dSmallSec} we use these to prove extensions of and improvements on Theorems \ref{NeatCoveringTheorem} and \ref{NeatPackingTheorem}.


\section{Notation and preliminary results}\label{prelimSec}


For a positive integer $v$, let $[v]$ denote the set $\{1,\ldots,v\}$. Let $J_{i}$ denote the $i \times i$ all-ones matrix. Let $G$ be a multigraph. All multigraphs in this paper are loopless.  For distinct $u,w \in V(G)$, we denote by $\mu_G(uw)$ the multiplicity of the edge $uw$. For $S \subseteq V(G)$, we denote by $G[S]$ the sub-multigraph of $G$ induced by $S$. The \emph{adjacency matrix} $A(G)$ of a multigraph $G$ with vertex set $[v]$ is the $v \times v$ matrix whose $uw$ entry is $\mu_G(uw)$ if $u \neq w$ and $0$ if $u=w$.

Let $\mathcal{D}$ be a $(v,k,\lambda)$-covering or -packing on point set $[v]$. For $u \in [v]$,
define $r_{\mathcal{D}}(u)$ to be the number of blocks of $\mathcal{D}$ containing $u$. Define a
multigraph $G$ on vertex set $[v]$ with $\mu_G(uw)=|r_{\mathcal{D}}(uw)-\lambda|$ for all distinct $u,w \in [v]$, where $r_{\mathcal{D}}(uw)$ is the number of blocks of $\mathcal{D}$ containing both $u$ and $w$. If $\mathcal{D}$ is a $(v,k,\lambda)$-covering then $G$ is called the \emph{excess} of $\mathcal{D}$, and if $\mathcal{D}$ is a $(v,k,\lambda)$-packing then $G$ is called the \emph{leave} of $\mathcal{D}$. Let $R$ be the diagonal matrix $\diag(r_{\mathcal{D}}(1)-\lambda,r_{\mathcal{D}}(2)-\lambda,\ldots,r_{\mathcal{D}}(v)-\lambda)$ and define $M(\mathcal{D})=R+A(G)$ if $\mathcal{D}$ is a $(v,k,\lambda)$-covering and $M(\mathcal{D})=R-A(G)$ if $\mathcal{D}$ is a $(v,k,\lambda)$-packing. Define $M^*(\mathcal{D})=M(\mathcal{D})+\lambda J_{v \times v}$.

We begin with the following observation which is a simple extension of the argument given in
the note \cite{Bos}.

\begin{lemma}\label{LIRowsLowerBound}
If $\mathcal{D}$ is a $(v,k,\lambda)$-covering or -packing on point set $[v]$, then
$\mathcal{D}$ has at least $\rank(M^*(\mathcal{D}))$ blocks.
\end{lemma}

\begin{proof}
Let $b$ be the number of blocks of  $\mathcal{D}$. Index the blocks of $\mathcal{D}$ with the elements of $[b]$ and let $X=(x_{uy})$ be the $v \times b$ matrix such that $x_{uy}=1$ if point $u$ is in block $y$ and $x_{uy}=0$ otherwise ($X$ is known as the incidence matrix of $\mathcal{D}$). It is not difficult to see that $XX^T=M^*(\mathcal{D})$. Thus, we have
$$b \geq \rank(X) \geq \rank(XX^T) = \rank(M^*(\mathcal{D})).$$\qedhere
\end{proof}

We now have a bound on the number blocks in a covering or packing $\mathcal{D}$ in terms of the rank of $M^*(\mathcal{D})$. In order to bound the rank of $M^*(\mathcal{D})$, we shall employ Lemma \ref{DominanceCor}. Lemma \ref{DominanceCor} is an easy consequence of the following well-known generalisation the Levy-Desplanques theorem (see \cite[Theorem IV]{Tau}, for example).

\newpage
\begin{lemma}[\cite{Tau}]\label{LDCor}
If $B = (b_{uw})$ is a $t \times t$ matrix with real entries such that, for each $u \in [t]$,
$$\medop\sum_{w \in [t] \setminus \{u\}} |b_{uw}| < b_{uu},$$
then $\det(B)>0$.
\end{lemma}

\begin{lemma}\label{DominanceCor}
If $A = (a_{uw})$ is a symmetric $s \times s$ matrix with real entries and there exist positive real numbers
$c_1,\ldots,c_s$ such that, for each $u \in [s]$,
$$\medop\sum_{w \in [s] \setminus \{u\}} c_w|a_{uw}| < c_ua_{uu},$$
then $A$ is positive definite.
\end{lemma}

\begin{proof}
By Sylvester's criterion it suffices to show that each leading principal minor of $A$ has positive determinant. Let $t \leq s$ be a positive integer and let $A_t$ be the $t$th leading principal minor of $A$. We show that $\det(A_t) > 0$. Let $B=(b_{uw})$ be the matrix obtained from $A_t$ by multiplying column $u$ by $c_u$ for
each $u \in [t]$ and note that $\det(B)=c_1\cdots c_s\det(A_t)$. Using our hypotheses, for each $u \in [t]$, we have
$$\medop\sum_{w \in [s] \setminus \{u\}} |b_{uw}| = \medop\sum_{w \in [s] \setminus \{u\}} c_w|a_{uw}| < c_ua_{uu} = b_{uu}.$$
So it follows from Lemma \ref{LDCor} that $\det(B)>0$ and hence that $\det(A_t)>0$.
\end{proof}

Note that the hypotheses of Lemma \ref{DominanceCor} can be weakened. In fact we only need to require
strict inequality for one row in each irreducible component of the matrix (see \cite{Tau} for
details). In certain specific cases this strengthening can be useful. To give a small example, it
can be used to show there does not exist a $(12,4,1)$-packing with nine blocks whose leave is a 12-cycle (if such a packing $\mathcal{D}$ existed then the matrix obtained from $M(\mathcal{D})$ by deleting a row would be positive definite and we could use an argument similar to the proof of Lemma \ref{MainLemma} below to show that $\mathcal{D}$ had at least eleven blocks). We will not require the stronger version for our purposes here, however.

\section{Basic bounds}\label{basicSec}

We introduce some more notation and note some basic facts about coverings and packings that we will use tacitly throughout the remainder of the paper.

Let $\mathcal{D}$ be a $(v,k,\lambda)$-covering or -packing on point set $[v]$ and let $G$ be the
excess or leave of $\mathcal{D}$. Define $b=b(\mathcal{D})$ to be the number of blocks of
$\mathcal{D}$. If $\mathcal{D}$ is a $(v,k,\lambda)$-covering, define $r=r(\mathcal{D})$
and $d=d(\mathcal{D})$ to be the integers such that $\lambda(v-1)=r(k-1)-d$ and $0 \leq d < k-1$,
and define $a=a(\mathcal{D})=bk-rv$. If $\mathcal{D}$ is a $(v,k,\lambda)$-packing, define
$r=r(\mathcal{D})$ and $d=d(\mathcal{D})$ to be the integers such that $\lambda(v-1)=r(k-1)+d$ and
$0 \leq d < k-1$, and define $a=a(\mathcal{D})=rv-bk$. Define $V_i=V_i(\mathcal{D})=\{u \in [v]:\deg_G(u)=d+i(k-1)\}$ for each nonnegative integer $i$. The following hold.
\begin{itemize}
    \item
For each nonnegative integer $i$ and each $u \in V_i$, $r_{\mathcal{D}}(u)=r+i$ if $\mathcal{D}$ is a $(v,k,\lambda)$-covering and $r_{\mathcal{D}}(u)=r-i$ if $\mathcal{D}$ is a $(v,k,\lambda)$-packing.
    \item
$\{V_0,V_1,\ldots\}$ is a partition of $[v]$.
    \item
$\sum_{u \in [v]}\deg_{G}(u)= dv+a(k-1)$.
    \item
$|[v] \setminus V_0| \leq a$ and $|V_0| \geq v-a$.
\end{itemize}

All of the results in this paper are based on the following lemma. It employs Lemma
\ref{DominanceCor} to obtain a bound on the number of blocks in a covering or packing based on the
structure of its excess or leave.

\begin{lemma}\label{MainLemma}
Let $v$, $k$ and $\lambda$ be positive integers such that $3 \leq k < v$, let $\mathcal{D}$ be a $(v,k,\lambda)$-covering or -packing on point set $[v]$, and let $G$ be the excess or leave of
$\mathcal{D}$. If there is a subset $S$ of $[v]$ and positive real numbers $(c_u)_{u
\in S}$ such that, for each $u \in S$,
$$\medop \sum_{w \in S \setminus \{u\}} c_w\mu_{G[S]}(uw) < c_u(r_{\mathcal{D}}(u) - \lambda)$$
then $\mathcal{D}$ has at least $|S|$ blocks.
\end{lemma}

\begin{proof}
Let $M(\mathcal{D}) = (m_{uw})$, let $s=|S|$, and let $A$ be the $s \times s$ submatrix of $M(\mathcal{D})$ containing only those rows and columns indexed by $S$. Note that $A$ is symmetric because $M(\mathcal{D})$ is symmetric by definition. By Lemma \ref{DominanceCor} $A$ is positive definite because, for each $u \in S$, we have
$$\medop\sum_{w \in S \setminus \{u\}} c_w|m_{uw}| = \medop \sum_{w \in S \setminus \{u\}} c_w\mu_{G[S]}(uw) < c_u(r_{\mathcal{D}}(u) - \lambda) = c_um_{uu}.$$
Thus the matrix $A+\lambda J_s$ is also positive definite, because the matrix $J_s$ is well known to be positive semi-definite. So $\rank(A+\lambda J_s)=s$ and, since $A+\lambda J_s$ is a submatrix of $M^*(\mathcal{D})$, $\rank(M^*(\mathcal{D})) \geq s$. The result now follows from Lemma \ref{LIRowsLowerBound}.
\end{proof}

In what follows we often choose $c_u=1$ for each $u \in S$ when applying Lemma
\ref{MainLemma}, and in these cases we will not make explicit mention of this choice when invoking the lemma. Note that a $(r'-\lambda)$-independent set $S'$ in $G$ where $r'\leq \min(\{r_\mathcal{D}(u):u \in S'\})$ is always a valid choice for $S$ ($m$-independence is defined in the next section).

It is now a relatively simple matter to prove Theorems \ref{NeatCoveringTheorem} and \ref{NeatPackingTheorem} which we restate here for convenience.

\begin{repCov}
Let $v$, $k$ and $\lambda$ be positive integers such that $3 \leq k <v$, and let $r$ and $d$ be the integers such that $\lambda(v-1)=r(k-1)-d$ and $0 \leq d < k-1$. If $d < r-\lambda$, then
$$C_{\lambda}(v,k) \geq \left\lceil\mfrac{v(r+1)}{k+1}\right\rceil.$$
\end{repCov}

\begin{repPac}
Let $v$, $k$ and $\lambda$ be positive integers such that $3 \leq k <v$, and let $r$ and $d$ be the integers such that $\lambda(v-1)=r(k-1)+d$ and $0 \leq d < k-1$. If $d < r-\lambda$, then
$$D_{\lambda}(v,k) \leq \left\lfloor\mfrac{v(r-1)}{k-1}\right\rfloor.$$
\end{repPac}

\begin{proof}[{\bf Proof of Theorems \ref{NeatCoveringTheorem} and \ref{NeatPackingTheorem}}]
Suppose that $\mathcal{D}$ is a $(v,k,\lambda)$-covering or -packing and let $G$ be the excess or
leave of $\mathcal{D}$. Note that $r=r(\mathcal{D})$ and $d=d(\mathcal{D})$. Let $b=b(\mathcal{D})$, $a=a(\mathcal{D})$ and $V_0=V_0(\mathcal{D})$. For each $u \in V_0$ we have $$\medop \sum_{w \in V_0 \setminus \{u\}} \mu_{G[V_0]}(uw) \leq \deg_G(u) = d < r-\lambda = r_{\mathcal{D}}(u) - \lambda.$$
Thus we can apply Lemma \ref{MainLemma} with $S=V_0$ to establish that $b \geq |V_0|$. Recall that $|V_0| \geq v-a$, so $b \geq v-a$. Applying the definition of $a$ and solving the resulting inequality for $b$ produces the required result.
\end{proof}

We compare the bound given by Theorem \ref{NeatCoveringTheorem} to the Sch{\"o}nheim bound. For a positive integer $i$, the bound given by Theorem \ref{NeatCoveringTheorem} will exceed the Sch{\"o}nheim bound by at least $i$ whenever $d < r-\lambda$, $k \geq 4i\lambda+5$ and $2i\lambda+1 \leq r \leq k-2i\lambda$. To see that
this is the case, observe that
$$\mfrac{v(r+1)}{k+1}-\mfrac{rv}{k} = \mfrac{v(k-r)}{k(k+1)} > \mfrac{(r-1)(k-r)(k-1)}{\lambda k(k+1)} \geq \mfrac{i(k+3)(k-1)}{k(k+1)} = \mfrac{i(k^2+2k-3)}{k^2+k}$$
and that this last expression is at least $i$ for $k \geq 3$. The first inequality holds because $\lambda v > (r-1)(k-1)$ and the second holds because $(r-1)(k-r) \geq i\lambda(k+3)$ which follows from $2i\lambda+1 \leq r \leq k-2i\lambda$ and $k \geq 4i\lambda+5$  (note that the former implies $(r-1)(k-r) \geq 2i\lambda(k-2i\lambda-1)$ and the latter implies $k-2i\lambda-1 \geq \frac{k+3}{2}$).

For a fixed $r$ in the range $2i\lambda+1 \leq r \leq k-2i\lambda$, there are at least
$\lfloor\frac{r-\lambda}{\lambda}\rfloor \geq \frac{r-2\lambda+1}{\lambda}$ values of $v$ such that $d<r-\lambda$. From this, it can be seen that for a given $k$ and $\lambda$, there are at least
$$\sum_{r=2i\lambda+1}^{k-2i\lambda}\mfrac{r-2\lambda+1}{\lambda}=\mfrac{(k-4i\lambda)(k-4\lambda+3)}{2\lambda}$$
integer values of $v$ for which Theorem \ref{NeatCoveringTheorem} improves the Sch{\"o}nheim bound by at least $i$. So, when $k$ is large in comparison with $i$ and $\lambda$, we obtain an improvement of $i$ or more for almost half of the less than $\frac{k^2}{\lambda}$ possible parameter sets for which $r<k$.

One interesting special case of Theorem \ref{NeatCoveringTheorem} to consider is the case where $\lambda v(v-1)+dv \equiv 0 \mod{k(k-1)}$ and hence a $(v,k,\lambda)$-covering meeting the Sch{\"o}nheim bound would necessarily have the same number of blocks on each point. In this case we have that $\frac{vr}{k}=\frac{\lambda v(v-1)+dv}{k(k-1)}$ is an integer and so the bound of Theorem \ref{NeatCoveringTheorem} exceeds the Sch{\"o}nheim bound by at least
$$\left\lceil\mfrac{v(r+1)}{k+1}\right\rceil - \mfrac{vr}{k} = \left\lceil\mfrac{v(k-r)}{k(k+1)}\right\rceil.$$
In particular, the bound is strictly greater than the Sch{\"o}nheim bound whenever $r<k$. Setting $d=0$ gives Fisher's inequality, setting $d=1$ yields a result of Bose and Connor \cite{BosCon}, and setting $d=2$ yields a result of Bryant, Buchanan, Horsley, Maenhaut and Scharaschkin \cite{BryBucHorMaeSch}. Table \ref{NeatCoveringImprovementsTable} gives examples of parameter sets for which Theorem \ref{NeatCoveringTheorem} strictly improves on the Sch{\"o}nheim bound. (For all tables in this paper, the maximum value of $k$ considered is determined only by space considerations.)

\begin{table}[H]
\begin{small}
\begin{center}
\begin{tabular}{|c|p{15.4cm}|}
\hline
$k$ & $v$ \\ \hline
$5$ & $17$ \\
$6$ & $20$, $21$, $24$, $25$, $26$ \\
$7$ & $23$, $24$, $25$, $28$, $29$, $30$, $31$, $35$, $36$, $37$ \\
$8$ & $27$, $28_{2}$, $29_{2}$, $33$, $34$, $35_{2}$, $36$, $39$, $40_{2}$, $41$, $42$, $43$, $48$, $ 49$, $50$ \\
$9$ & $31_{2}$, $32$, $33_{2}$, $38$, $39_{2}$, $40$, $41_{2}$, $45_{2}$, $46_{2}$, $47$, $48_{2}$, $ 49_{2}$, $52$, $53$, $54_{2}$, $55$, $56$, $57$, $59$, $63$, $64$, $65$ \\
$10$ & $35_{2}$, $36_{2}$, $37_{2}$, $43_{2}$, $44_{2}$, $45_{2}$, $46_{3}$, $51_{2}$, $52_{2}$, $53_{2}$, $54_{2}$, $55_{2}$, $59$, $60_{2}$, $61_{2}$, $62_{2}$, $63$, $64_{2}$, $67$, $68$, $69$, $ 70_{2}$, $71_{2}$, $72$, $73$, $75$, $76$, $80$, $81$, $82$ \\
$11$ & $39_{2}$, $40_{2}$, $41_{3}$, $48_{2}$, $49_{2}$, $50_{2}$, $51_{2}$, $57_{2}$, $58_{2}$, $ 59_{2}$, $60_{2}$, $61_{2}$, $66_{2}$, $67_{2}$, $68_{2}$, $69_{2}$, $70_{2}$, $71_{2}$, $75_{2},$ $76$, $77_{2}$, $78_{2}$, $79_{2}$, $80$, $81_{2}$, $84$, $85$, $86$, $87$, $88_{2}$, $89_{2}$, $90$, $91$, $93$, $94$, $95$, $99$, $100$, $101$ \\
$12$ & $43_{2}$, $44_{2}$, $45_{3}$, $53_{2}$, $54_{2}$, $55_{3}$, $56_{2}$, $63_{2}$, $64_{3}$, $ 65_{2}$, $66_{3}$, $67_{3}$, $73_{2}$, $74_{2}$, $75_{3}$, $76_{2}$, $77_{3}$, $78_{2}$, $83_{2},$ $84_{3}$, $85_{2}$, $86_{2}$, $87_{3}$, $88_{2}$, $89_{2}$, $93_{2}$, $94_{2}$, $95_{2}$, $96_{2}$, $ 97_{2}$, $98_{2}$, $99_{2}$, $100_{2}$, $103_{2}$, $104$, $105$, $106$, $107$, $108_{2}$, $109_{2}$, $110_{2}$, $111$, $113$, $114$, $115$, $116$, $120$, $121$, $122$ \\
 \hline
\end{tabular}
\end{center}
\end{small}
\vspace{-0.6cm}

\caption{For $\lambda=1$ and each $k\in \{3,\ldots,12\}$, the values of $v > \frac{13}{4}k$ for which Theorem \ref{NeatCoveringTheorem} strictly improves on the Sch{\"o}nheim bound. Values of $v$ for which the Sch{\"o}nheim bound is improved by $i \geq 2$ are marked with a subscript $i$.}
\label{NeatCoveringImprovementsTable}
\end{table}

\section{$m$-independent sets}\label{indepSec}

An edge-weighted graph $G$ is a complete (simple) graph whose edges have been assigned nonnegative real weights. We represent the weight of an edge $uw$ in such a graph $G$ by $\wt_G(uw)$ and we define the weight of a vertex $u$ of $G$ as $\wt_G(u)=\sum_{w \in V(G) \setminus \{u\}}\wt_G(uw)$. For $S \subseteq V(G)$, we denote by $G[S]$ the edge-weighted subgraph of $G$ induced by $S$. If $m$ is a positive integer and $G$ is an edge-weighted graph, then a subset $S$ of $V(G)$ is said to be an \emph{$m$-independent set} in $G$ if $\wt_{G[S]}(u) < m$ for each $u \in S$. An algorithm for finding an $m$-independent set in an edge-weighted graph, which we shall call $m$-MAX, operates by beginning with the graph and iteratively deleting an (arbitrarily chosen) vertex of maximum weight in the remaining graph until all the vertex weights in the remaining graph are less than $m$. The vertices of this subgraph form an $m$-independent set in the original graph.

A multigraph can be represented as an edge-weighted graph whose edge and vertex weights correspond to the multiplicities of edges and degrees of vertices in the original multigraph. Viewing multigraphs in this way, we recover the usual definitions of an $m$-independent set and the algorithm $m$-MAX from the definitions in the preceding paragraph.

Caro and Tuza \cite{CarTuz} established a lower bound on the size of an $m$-independent set yielded by an application of $m$-MAX to a multigraph in terms of the degree sequence of the multigraph. Lemma \ref{CTVariant} below is an adaptation of this result to the setting of edge-weighted graphs. Its proof requires no new ideas and follows the proof given in \cite{CarTuz} closely. For a positive integer $m$, define a function $f_m: \{x \in \mathbb{R}:x\geq0\} \rightarrow \{x \in \mathbb{R}:0<x\leq1\}$ by
$$f_m(x)=\left\{
         \begin{array}{ll}
           1-\tfrac{x}{2m}, & \hbox{if $x \leq m$;} \\
           \tfrac{m+1}{2x+2}, & \hbox{if $x \geq m$.}
         \end{array}
       \right.$$
It can be seen that $f_m$ has the following properties.
\begin{itemize}
    \item[(F1)]
$f_m$ is continuous, convex, and monotonically decreasing.
    \item[(F2)]
$f_m(x-y)-f_m(x) \geq \frac{y(m+1)}{2x(x+1)}$ for any real numbers $x$ and $y$ with
$x \geq m$ and $1 \leq y \leq x$.
\end{itemize}
To see that (F2) holds, observe that from the definition of $f_m$ we have
$$f_m(x-y)-f_m(x)-\mfrac{y(m+1)}{2x(x+1)}=
    \left\{
    \begin{array}{ll}
        \mfrac{y(m+1)}{2(x+1)(x+1-y)}-\mfrac{y(m+1)}{2x(x+1)}, & \hbox{if $x-y \geq m$;} \\[0.3cm]
        \mfrac{(x+1)(x-m)(m+y-x)+m(y-1)(x-m)}{2mx(x+1)}, & \hbox{if $x-y \leq m$;}
    \end{array}
    \right.$$
and that this is nonnegative, using the facts that $x \geq m$ and $1 \leq y \leq x$.

\begin{lemma}\label{CTVariant}
Let $m$ be a positive integer and let $G$ be an edge-weighted graph in which any edge incident with two vertices of weight at least $m$ has weight at least $1$. Then any application of $m$-MAX to $G$ will yield an $m$-independent set in $G$ of size at least $\lceil\sum_{u \in V(G)} f_m(\wt_G(u))\rceil$.
\end{lemma}

\begin{proof}
Let $G$ be a fixed edge-weighted graph and let $F=\sum_{u \in V(G)} f_m(\wt_G(u))$. If $G$ has
only one vertex, then $F=f_m(0)=1$ and the result clearly holds. Suppose by induction that the result holds for all edge-weighted graphs with fewer vertices than $G$.

Let $w$ be an arbitrary vertex of maximum weight in $G$. We may suppose that $\wt_G(w) \geq m$, for otherwise $V(G)$ is $m$-independent in $G$ and, since $f_m(\wt_G(u)) \leq 1$ for each $u \in G$, we are finished immediately. Let $G'$ be the graph obtained from $G$ by deleting $w$ and all edges incident with $w$, and let $F'=\sum_{u \in V(G')} f_m(\wt_{G'}(u))$. If $F' \geq F$ then, applying our inductive hypothesis, we see that any application of $m$-MAX to $G'$ will yield an $m$-independent set of size at least $\lceil F' \rceil \geq \lceil F \rceil$. Thus, because $w$ was chosen arbitrarily, any application of $m$-MAX to $G$ will yield an $m$-independent set of size at least $\lceil F \rceil$. So it suffices to show that $F' \geq F$.

For nonnegative real numbers $x$ and $y$ with $y \leq x$, let $f^*_m(x,y)=f_m(x-y)-f_m(x)$. It can
be seen that
$$F'-F=\left(\medop \sum_{u\in V(G) \setminus \{w\}} f^*_m(\wt_G(u),\wt_G(uw))\right)-f_m(\wt_G(w)).$$
So, noting that $f_m(\wt_G(w))=\frac{m+1}{2\wt_G(w)+2}$ and that $\wt_G(w)=\sum_{u\in
V(G) \setminus \{w\}}\wt_G(uw)$, it in fact suffices to show that, for each $u \in V(G)$,
\begin{equation}\label{indepReduction}
f^*_m(\wt_G(u),\wt_G(uw)) \geq \left(\mfrac{\wt_G(uw)}{\wt_G(w)}\right)\left(\mfrac{m+1}{2\wt_G(w)+2}\right).
\end{equation}
If $u$ is a vertex of $G$ with $\wt_G(u)<m$, then using the definition of $f_m$ we have
$f^*_m(\wt_G(u),\wt_G(uw))=\frac{\wt_G(uw)}{2m}$ and hence \eqref{indepReduction} holds because
$\wt_G(w) \geq m$. If $u$ is a vertex of $G$ with $\wt_G(u) \geq m$, then $\wt_G(uw) \geq 1$ from our hypotheses and thus, using Property (F2) of $f_m$, we have $f^*_m(\wt_G(u),\wt_G(uw)) \geq
(\frac{\wt_G(uw)}{\wt_G(u)})(\frac{m+1}{2\wt_G(u)+2})$. So again \eqref{indepReduction} holds
because $\wt_G(w) \geq \wt_G(u)$.
\end{proof}

\begin{lemma}\label{CTVariantCor}
Let $m$ be a positive integer and let $G$ be an edge-weighted graph in which any edge incident with two vertices of weight at least $m$ has weight at least $1$. The following hold
\begin{itemize}
    \item[(a)]
For any nonempty subset $S$ of $V(G)$, any application of $m$-MAX to $G[S]$ will yield an $m$-independent set in $G[S]$ of size at least $\lceil|S|f_m(x)\rceil$ where $x=\frac{1}{|S|}\sum_{u \in S}\wt_G(u)$.
    \item[(b)]
For any two disjoint nonempty subsets $S_0$ and $S_1$ of $V(G)$, any application of $m$-MAX to $G[S_0 \cup S_1]$ will yield an $m$-independent set in $G[S_0 \cup S_1]$ of size at least
$\lceil|S_0|f_m(x_0)+|S_1|f_m(x_1)\rceil$ where $x_i=\frac{1}{|S_i|}\sum_{u \in S_i}\wt_G(u)$ for $i \in \{0,1\}$.
\end{itemize}
\end{lemma}

\begin{proof}
We will prove (a). The proof of (b) is similar. From (F1) we know that $f_m$ is convex and monotonically decreasing. Let $S$ be a subset of $V(G)$. By Lemma \ref{CTVariant}, any application of $m$-MAX to $G[S]$ will yield an $m$-independent set in $G[S]$ of size at least $\lceil F \rceil$ where $F=\sum_{u \in S} f_m(\wt_{G[S]}(u))$. For any $u \in S$ we have $f_m(\wt_{G[S]}(u)) \geq f_m(\wt_{G}(u))$ because $\wt_{G[S]}(u) \leq \wt_{G}(u)$ and $f_m$ is monotonically decreasing. Thus,
$$F \geq \medop\sum_{u \in S} f_m(\wt_{G}(u)) \geq |S|f_m\left(\mfrac{1}{|S|}\medop\sum_{u \in S}\wt_G(u)\right)$$
where the second inequality follows from the convexity of $f_m$.
\end{proof}

\section{Bounds for the case $d \geq r-\lambda$}\label{dBigSec}

We require some further definitions to state our subsequent bounds concisely. For positive integers $v$, $k$ and $\lambda$ such that $3 \leq k <v$ and nonnegative real numbers $\alpha$ and $\beta$ such that $\alpha \geq \beta$, we define
$$CB_{(v,k,\lambda)}(\alpha,\beta) = \mfrac{rv(\alpha-\beta)+\alpha v}{k(\alpha-\beta)+1}, \quad \mbox{where} \quad r=\left\lceil\mfrac{\lambda(v-1)}{k-1}\right\rceil;$$
and, if $\alpha>\beta+\frac{1}{k}$,
$$DB_{(v,k,\lambda)}(\alpha,\beta) = \mfrac{rv(\alpha-\beta)-\alpha v}{k(\alpha-\beta)-1}, \quad \mbox{where} \quad r=\left\lfloor\mfrac{\lambda(v-1)}{k-1}\right\rfloor.$$
Note that the bounds given by Theorems \ref{NeatCoveringTheorem} and \ref{NeatPackingTheorem} are $\lceil CB_{(v,k,\lambda)}(1,0) \rceil$ and $\lfloor DB_{(v,k,\lambda)}(1,0) \rfloor$ respectively.
The next two results are technical lemmas that allow us to establish that $C_{\lambda}(v,k) \geq \lceil CB_{(v,k,\lambda)}(\alpha,\beta) \rceil$ and $D_{\lambda}(v,k) \leq \lfloor DB_{(v,k,\lambda)}(\alpha,\beta) \rfloor$ for certain values of $\alpha$ and $\beta$.

\begin{lemma}\label{BasicCoveringCalcLemma}
Let $v$, $k$ and $\lambda$ be positive integers such that $3 \leq k <v$. Suppose that any $(v,k,\lambda)$-covering $\mathcal{D}$ has at least $\alpha|V_0(\mathcal{D})|+\beta|V_1(\mathcal{D})|$ blocks, where $\alpha$ and $\beta$ are nonnegative real numbers such that $\alpha \geq 2\beta$. Then $C_{\lambda}(v,k) \geq \lceil CB_{(v,k,\lambda)}(\alpha,\beta) \rceil$.
\end{lemma}

\begin{lemma}\label{BasicPackingCalcLemma}
Let $v$, $k$ and $\lambda$ be positive integers such that $3 \leq k <v$. Suppose that any $(v,k,\lambda)$-packing $\mathcal{D}$ has at least $\alpha|V_0(\mathcal{D})|+\beta|V_1(\mathcal{D})|$ blocks, where $\alpha$ and $\beta$ are nonnegative real numbers such that $\alpha \geq 2\beta$ and $\alpha >\beta+\frac{1}{k}$. Then $D_{\lambda}(v,k) \leq \lfloor DB_{(v,k,\lambda)}(\alpha,\beta) \rfloor$.
\end{lemma}

\begin{proof}[{\bf Proof of Lemmas \ref{BasicCoveringCalcLemma} and \ref{BasicPackingCalcLemma}}]
Suppose that $\mathcal{D}$ is a $(v,k,\lambda)$-covering or -packing on point set $[v]$ and let $G$ be the excess or leave of $\mathcal{D}$.  Let $b=b(\mathcal{D})$, $r=r(\mathcal{D})$, $d=d(\mathcal{D})$, $a=a(\mathcal{D})$, $V_i=V_i(\mathcal{D})$ for $i \in \{0,1\}$, and $v_i=|V_i|$ for $i \in \{0,1\}$. Note that $v_1+2(v-v_0-v_1) \leq a$ because $\deg_{G}(u) = d+i(k-1)$ for each $u \in V_i$ for $i \in \{0,1\}$, $\deg_{G}(u) \geq d+2(k-1)$ for each $u \in V \setminus (V_0 \cup V_1)$, and $\sum_{u \in [v]}\deg_G(u)=dv+a(k-1)$. It follows that $v_0 \geq \frac{1}{2}(2v-v_1-a)$ and so from our hypotheses we have
$$b \geq \tfrac{1}{2}\alpha(2v-v_1-a)+\beta v_1 = \tfrac{1}{2}\alpha(2v-a)-\tfrac{1}{2}(\alpha-2\beta)v_1.$$
Thus, because $\alpha \geq 2\beta$, it follows from $v_1 \leq |[v] \setminus V_0| \leq a$ that
$$b \geq \tfrac{1}{2}\alpha(2v-a)-\tfrac{1}{2}(\alpha-2\beta)a = \alpha v-(\alpha-\beta)a.$$
Applying the definition of $a$ and solving the resulting inequality for $b$ produces the required result (note that $\alpha > \beta+\frac{1}{k}$ if $\mathcal{D}$ is a $(v,k,\lambda)$-packing).
\end{proof}

\begin{theorem}\label{d>rCoveringTheorem}
Let $v$, $k$ and $\lambda$ be positive integers such that $3 \leq k <v$, let $r$ and $d$ be the integers such that $\lambda(v-1)=r(k-1)-d$ and $0 \leq d < k-1$, let $n=r-\lambda$, and suppose that $r<k$. If $d \geq n$, then
$$C_{\lambda}(v,k) \geq \lceil
CB_{(v,k,\lambda)}(\tfrac{n+1}{2d+2},\tfrac{n+1}{2(d+k)})\rceil.$$
\end{theorem}

\begin{theorem}\label{d>rPackingTheorem}
Let $v$, $k$ and $\lambda$ be positive integers such that $3 \leq k <v$, let $r$ and $d$ be the integers such that $\lambda(v-1)=r(k-1)+d$ and $0 \leq d < k-1$, let $n=r-\lambda$, and suppose that $r<k$. If $d \geq n$, then
\begin{itemize}
    \item[(a)]
$D_{\lambda}(v,k) \leq \lfloor DB_{(v,k,\lambda)}(\tfrac{n+1}{2d+2},0)\rfloor$ if
$k(n+1) > 2d+2$; and
    \item[(b)]
$D_{\lambda}(v,k) \leq \lfloor
DB_{(v,k,\lambda)}(\tfrac{n}{2d+2},\tfrac{n}{2(d+k)})\rfloor$ if
$nk(k-1) > 2(d+1)(d+k)$.
\end{itemize}
\end{theorem}

\begin{proof}[{\bf Proof of Theorems \ref{d>rCoveringTheorem} and \ref{d>rPackingTheorem}}]
Suppose that $\mathcal{D}$ is a $(v,k,\lambda)$-covering or -packing and let $G$ be the excess or
leave of $\mathcal{D}$. Note that $r=r(\mathcal{D})$ and $d=d(\mathcal{D})$. Let $b=b(\mathcal{D})$, $a=a(\mathcal{D})$, $V_i=V_i(\mathcal{D})$ for $i \in \{0,1\}$ and $v_i=|V_i|$ for $i \in \{0,1\}$.

\noindent{\bf Bound \ref{d>rPackingTheorem}(a).} Observe that if $S$ is an $n$-independent set in $G[V_0]$ then $b \geq |S|$ by Lemma \ref{MainLemma}. By Lemma \ref{CTVariantCor}(a), $G[V_0]$ has an $n$-independent set of size at least $|V_0|f_{n}(d)$ where $f_{n}(d)=\frac{n+1}{2d+2}$ and hence
$$b \geq \mfrac{n+1}{2d+2}|V_0|.$$
Applying Lemma \ref{BasicPackingCalcLemma} with $\alpha=\frac{n+1}{2d+2}$ and $\beta=0$ yields the desired bound (note that clearly $\alpha \geq 2\beta$ and that $k(n+1) > 2d+2$ implies $\alpha > \beta+\frac{1}{k}$).

\noindent{\bf Bounds \ref{d>rCoveringTheorem} and \ref{d>rPackingTheorem}(b).} Let $m=n$ if $\mathcal{D}$ is a $(v,k,\lambda)$-covering and
$m=n-1$ if $\mathcal{D}$ is a $(v,k,\lambda)$-packing. Observe that if $S$ is an $m$-independent set in $G[V_0 \cup V_1]$, then $b \geq |S|$ by Lemma \ref{MainLemma}. By Lemma \ref{CTVariantCor}(b), $G[V_0 \cup V_1]$ has an $m$-independent set of size at least $|V_0|f_{m}(d)+|V_1|f_{m}(d+k-1)$ where $f_{m}(d)=\frac{m+1}{2d+2}$ and $f_{m}(d+k-1)=\frac{m+1}{2(d+k)}$ and hence
    $$b \geq \mfrac{m+1}{2d+2}|V_0|+\mfrac{m+1}{2(d+k)}|V_1|.$$
Applying Lemma \ref{BasicCoveringCalcLemma} or \ref{BasicPackingCalcLemma} with $\alpha=\frac{m+1}{2d+2}$ and $\beta=\frac{m+1}{2(d+k)}$ yields the appropriate bound (note that $\alpha \geq 2\beta$ because $k \geq d+2$ and, if $\mathcal{D}$ is a $(v,k,\lambda)$-packing, that $nk(k-1) > 2(d+1)(d+k)$ implies $\alpha > \beta+\frac{1}{k}$).
\end{proof}

It is never the case that both Theorems \ref{NeatCoveringTheorem} and \ref{d>rCoveringTheorem} or both Theorems \ref{NeatPackingTheorem} and \ref{d>rPackingTheorem} apply to the same parameter set because Theorems \ref{NeatCoveringTheorem} and \ref{NeatPackingTheorem} require $d < r-\lambda$ and Theorems \ref{d>rCoveringTheorem} and \ref{d>rPackingTheorem} require $d \geq r-\lambda$. Note that there are some parameter sets for which the bound of Theorem~\ref{d>rPackingTheorem}(a) is smaller than the bound of Theorem \ref{d>rPackingTheorem}(b) and others for which the reverse is true. We now compare the bound of Theorem \ref{d>rCoveringTheorem} to the Sch{\"o}nheim bound. Observe that, for real numbers $\alpha$ and $\beta$ such that $\alpha \geq \beta > 0$, we have
\begin{equation}\label{d>rThComp}
CB_{(v,k,\lambda)}(\alpha,\beta)-\mfrac{rv}{k}=\mfrac{v(k\alpha-r)}{k(k(\alpha-\beta)+1)}.
\end{equation}
Note that if $C_{\lambda}(v,k) \geq CB_{(v,k,\lambda)}(\alpha,\beta)$, then $C_{\lambda}(v,k) \geq CB_{(v,k,\lambda)}(\alpha,\beta')$ for any $0 \leq \beta' \leq \beta$. This is because $CB_{(v,k,\lambda)}(\alpha,\beta)$ is monotonically increasing in $\beta$ when $k\alpha>r$, and $CB_{(v,k,\lambda)}(\alpha,\beta')$ is at most the Sch{\"o}nheim bound when $k\alpha \leq r$. Setting $\alpha=\frac{n+1}{2d+2}$ and $\beta=\frac{n+1}{2(d+k)}$, we see that the bound of Theorem \ref{d>rCoveringTheorem} will match or exceed the Sch{\"o}nheim bound whenever $k(n+1)>2r(d+1)$.

Infinite families of parameter sets for which the bounds of Theorem \ref{d>rCoveringTheorem} yield arbitrarily large improvements on the Sch{\"o}nheim bound can be found. Suppose that $\lambda$ is constant and $k \rightarrow \infty$. When $\alpha=\frac{n+1}{2d+2}$ and $\beta = \frac{n+1}{2(d+k)}$, \eqref{d>rThComp} implies that
$$CB_{(v,k,\lambda)}(\alpha,\beta)-\tfrac{rv}{k}=\Omega(\tfrac{r^2}{dk})(k-2d-2)-O(1),$$
noting that $v = \Theta(kr)$, and that $\alpha-\beta = O(1)$ because $\alpha \leq \frac{1}{2}$ and $2\alpha \geq \beta$. So, for example, if $k-2d \rightarrow \infty$, $d  \geq r-\lambda$ and $r = \Theta(k)$, we will obtain arbitrarily large improvements on the Sch{\"o}nheim bound. Table \ref{d>rCoveringImprovementsTable} gives examples of parameter sets for which Theorem \ref{d>rCoveringTheorem} strictly improves on the Sch{\"o}nheim bound.

\begin{table}[H]
\begin{small}
\begin{center}
\begin{tabular}{|c|l|}
\hline
$k$ & $v$ \\ \hline
$10$ & $34$ \\
$11$ & $38$ \\
$12$ & $41, 42_{2}, 52$ \\
$13$ & $45, 46, 57_{2}$ \\
$14$ & $48, 49_{2}, 50, 61, 62, 74$ \\
$15$ & $52, 53, 54_{2}, 65, 66_{2}, 67_{2}, 79, 80_{2}$ \\
$16$ & $55, 56, 57, 58_{2}, 70, 71, 72_{2}, 85, 86, 100$ \\
$17$ & $59, 60, 61, 62_{2}, 74, 76, 77_{2}, 90, 91, 92_{2}, 106, 107$ \\
$18$ & $62, 63, 64, 65_{2}, 66_{2}, 79, 80, 81_{2}, 82_{3}, 96, 97, 98_{2}, 113, 114_{2}, 130$ \\
$19$ & $66, 67, 68, 69_{2}, 70_{3}, 83, 85, 86_{2}, 87_{3}, 101, 102, 103_{2}, 104_{3}, 119, 120, 121_{2}, 137, 138$ \\
$20$ & $69, 70, 71, 72, 73_{2}, 74_{3}, 88, 89, 90_{2}, 91_{2}, 92_{3}, 108, 109_{2}, 110_{3}, 127_{2}, 128_{3}, 145, 146_{2}, 164$ \\
\hline
\end{tabular}
\end{center}
\end{small}
\vspace{-0.6cm}

\caption{For $\lambda=1$ and each $k\in \{3,\ldots,20\}$, the values of $v > \frac{13}{4}k$ for which Theorem \ref{d>rCoveringTheorem} strictly improves on the Sch{\"o}nheim bound. Values of $v$ for which the Sch{\"o}nheim bound is improved by $i \geq 2$ are marked with a subscript $i$.}
\label{d>rCoveringImprovementsTable}
\end{table}

\section{More bounds for the case $d<r-\lambda$}\label{dSmallSec}

\begin{lemma}\label{IndepSetCase2}
Let $v$, $k$ and $\lambda$ be positive integers such that $3 \leq k < v$, let $\mathcal{D}$ be a $(v,k,\lambda)$-covering or -packing, and let $G$ be the excess or leave of
$\mathcal{D}$. Let $r=r(\mathcal{D})$, let $d=d(\mathcal{D})$, let $V_i=V_i(\mathcal{D})$ for $i \in \{0,1\}$, and suppose $d<r-\lambda$. Let $m=r-\lambda+1$ if $\mathcal{D}$ is a $(v,k,\lambda)$-covering and $m=r-\lambda-1$ if $\mathcal{D}$ is a $(v,k,\lambda)$-packing. Let $c$ be a real number such that $\frac{d}{r-\lambda} < c < 1$ and let $G^*$ be the edge-weighted graph on vertex set $V_0 \cup V_1$ such that
\begin{itemize}
    \item
$\wt_{G^*}(uw)=\mu_G(uw)$ for all distinct $u,w \in V_1$;
    \item
$\wt_{G^*}(uw)=c\mu_G(uw)$ for all $u \in V_0$, $w \in V_1$; and
    \item
$\wt_{G^*}(uw)=0$ for all distinct $u,w \in V_0$.
\end{itemize}
If $S$ is an $m$-independent set in $G^*$, then $\mathcal{D}$ has at least $|S|$ blocks.
\end{lemma}

We call the graph $G^*$ in Lemma \ref{IndepSetCase2} the \emph{$c$-reduced excess} or
\emph{$c$-reduced leave} of $\mathcal{D}$.

\begin{proof}
Let $S$ be an $m$-independent set in $G^*$ and let $S_i=S \cap V_i$ for $i \in \{0,1\}$. We show that we can apply Lemma \ref{MainLemma} to $G[S]$ choosing $c_u=c$ for $u \in S_0$ and $c_u=1$ for $u \in S_1$. This will suffice to complete the proof.

If $u \in S_0$, then $c_u=c$, $r_{\mathcal{D}}(u)-\lambda = r-\lambda$, and
$$\medop \sum_{w \in S \setminus \{u\}} c_w\mu_{G[S]}(uw) \leq \deg_{G[S]}(u) \leq d < c(r-\lambda)=c_u(r_{\mathcal{D}}(u)-\lambda).$$
If $u \in S_1$, then $c_u=1$, $r_{\mathcal{D}}(u)-\lambda = m$, and
$$\medop \sum_{w \in S \setminus \{u\}} c_w\mu_{G[S]}(uw)
= \medop \sum_{w \in S_1 \setminus \{u\}}\mu_{G[S]}(uw)+c\medop \sum_{w \in S_0} \mu_{G[S]}(uw) = \wt_{G^*[S]}(u) < m = c_u(r_{\mathcal{D}}(u)-\lambda),$$
where the inequality follows from the fact that $S$ is an $m$-independent set in $G^*$.
\end{proof}

\begin{theorem}\label{d<rTrickyCoveringTheorem}
Let $v$, $k$ and $\lambda$ be positive integers such that $3 \leq k <v$, let $r$ and $d$ be the integers such that $\lambda(v-1)=r(k-1)-d$ and $0 \leq d < k-1$, let $n=r-\lambda$, and suppose that $r<k$. If $d < n$, then
\begin{itemize}
    \item[(a)]
$C_\lambda(v,k) \geq \left\lceil CB_{(v,k,\lambda)}\left(1-\frac{d^2}{2n(n+1)},\frac{n+2}{2(d+k)}\right)\right\rceil$;
    \item[(b)]
$C_\lambda(v,k) \geq \left\lceil CB_{(v,k,\lambda)}\left(1,1-\frac{d(d+k-1)}{n(n+1)}\right)\right\rceil$ if $d \geq
\frac{n}{2}$ and $d(d+k-1) < n(n+1)$; and
    \item[(c)]
$C_\lambda(v,k) \geq \left\lceil
CB_{(v,k,\lambda)}\left(1,\sqrt{\frac{d(n+2)}{(n+1)(n-d)}}-\frac{d(d+k)}{2(n+1)(n-d)}\right)\right\rceil$ if $d < \frac{n}{2}$ and $4(n+1)(n+2)(n-d) > d(d+k)^2$.
\end{itemize}
\end{theorem}

\begin{theorem}\label{d<rTrickyPackingTheorem}
Let $v$, $k$ and $\lambda$ be positive integers such that $3 \leq k <v$, let $r$ and $d$ be the integers such that $\lambda(v-1)=r(k-1)+d$ and $0 \leq d < k-1$, let $n=r-\lambda$, and suppose that $r<k$. If $d < n$, then
\begin{itemize}
    \item[(a)]
$D_\lambda(v,k) \leq \left\lfloor DB_{(v,k,\lambda)}\left(1-\frac{d^2}{2n(n-1)},\frac{n}{2(d+k)}\right)\right\rfloor$;
    \item[(b)]
$D_\lambda(v,k) \leq \left\lfloor DB_{(v,k,\lambda)}\left(1,1-\frac{d(d+k-1)}{n(n-1)}\right)\right\rfloor$ if $d \geq
\frac{n}{2}$ and $d(d+k-1) < n(n-1)$; and
    \item[(c)]
$D_\lambda(v,k) \leq \left\lfloor
DB_{(v,k,\lambda)}\left(1,\sqrt{\frac{dn}{(n-1)(n-d)}}-\frac{d(d+k)}{2(n-1)(n-d)}\right)\right\rfloor$ if $d <
\frac{n}{2}$ and $4n(n-1)(n-d) > d(d+k)^2$.
\end{itemize}
\end{theorem}

\begin{proof}
Suppose that $\mathcal{D}$ is a $(v,k,\lambda)$-covering or -packing and let $G$ be the excess or
leave of $\mathcal{D}$. Note that $r=r(\mathcal{D})$ and $d=d(\mathcal{D})$. Let $b=b(\mathcal{D})$, $a=a(\mathcal{D})$, $V_i=V_i(\mathcal{D})$ for $i \in \{0,1\}$ and $v_i=|V_i|$ for $i \in \{0,1\}$. Let $m=n+1$ if $\mathcal{D}$ is a $(v,k,\lambda)$-covering and $m=n-1$ if
$\mathcal{D}$ is a $(v,k,\lambda)$-packing. It follows from these definitions and from $r<k$ that $k \geq m+1$. Let
$$(\alpha_{\rm a},\beta_{\rm a})=\left(1-\tfrac{d^2}{2mn},\tfrac{m+1}{2(d+k)}\right), \left(\alpha_{\rm b},\beta_{\rm b}\right)=\left(1,1-\tfrac{d(d+k-1)}{mn}\right),  (\alpha_{\rm c},\beta_{\rm c})= \left(1,\sqrt{\tfrac{d(m+1)}{m(n-d)}}-\tfrac{d(d+k)}{2m(n-d)}\right).$$

Note that bounds (a), (b) and (c) of the appropriate theorem can be obtained by applying Lemma \ref{BasicCoveringCalcLemma} or \ref{BasicPackingCalcLemma} with $(\alpha,\beta)$ chosen to be $(\alpha_{\rm a},\beta_{\rm a})$, $(\alpha_{\rm b},\beta_{\rm b})$ and $(\alpha_{\rm c},\beta_{\rm c})$ respectively. So it suffices to show that we can apply Lemma \ref{BasicCoveringCalcLemma} or \ref{BasicPackingCalcLemma} in these cases.

We first show, for each $i \in \{{\rm a},{\rm b},{\rm c}\}$ that $\alpha_i \geq 2\beta_i >0$ and that $\alpha_i > \beta_i+\frac{1}{k}$ if $\mathcal{D}$ is a $(v,k,\lambda)$-packing. It is easy to check from the hypotheses and conditions of the appropriate theorem that $\beta_i >0$ for each $i \in \{{\rm a},{\rm b},{\rm c}\}$.

\noindent{\bf Case (a).} Note that
$$\alpha_{\rm a}-2\beta_{\rm a}=1-\mfrac{d^2}{2mn} - \mfrac{m+1}{d+k} = \mfrac{2n(k-m-1)(m-d)+d(n(k-2)-d^2)+dk(n-d)}{2mn(d+k)}$$
and that the latter expression is nonnegative since $k > m \geq d$, $k \geq d+2$ and $n >d$. Also note that, if $\mathcal{D}$ is a $(v,k,\lambda)$-packing,
$$\alpha_{\rm a}-\beta_{\rm a}-\tfrac{1}{k}=1-\mfrac{d^2}{2mn} - \mfrac{m+1}{2(d+k)} - \mfrac{1}{k} = \mfrac{k(k+d)(mn-d^2)+kmn(k-m-3)+dmn(k-2)}{2kmn(d+k)}$$
and that the latter expression is positive since $n > d$ and $k-3\geq m \geq d$ (to see that $k-3\geq m$, note that $r\leq k-1$ and that $m = n-1\leq r-2$ since $\mathcal{D}$ is a packing).

\noindent{\bf Case (b).} Suppose that $d \geq \frac{n}{2}$. Using this and the fact that $k \geq m+1$, we have
$$\beta_{\rm b} = 1-\mfrac{d(d+k-1)}{mn} \leq 1-\mfrac{(d+k-1)}{2m} \leq 1-\mfrac{d+m}{2m} < \mfrac{1}{2}.$$
Thus $\alpha_{\rm b} \geq 2\beta_{\rm b}$ and, if $\mathcal{D}$ is a $(v,k,\lambda)$-packing, $\alpha_{\rm b} > \beta_{\rm b}+\frac{1}{k}$.

\noindent{\bf Case (c).} Note that $\frac{d(m+1)}{m(n-d)} < \frac{d(d+k)}{m(n-d)}$ because $k \geq m+1$. So, because $\sqrt{x}-\frac{x}{2} \leq \frac{1}{2}$ for any nonnegative real number $x$, we have that $\beta_{\rm c} \leq \frac{1}{2}$. Thus $\alpha_{\rm c} \geq 2\beta_{\rm c}$ and if $\mathcal{D}$ is a $(v,k,\lambda)$-packing, then $\alpha_{\rm c} > \beta_{\rm c}+\frac{1}{k}$.

It remains to show that $\mathcal{D}$ has at least $\alpha_i v_0+\beta_i v_1$ blocks for each $i \in \{{\rm a},{\rm b},{\rm c}\}$ (note that we have just shown that $\beta_i>0$ for each $i \in \{{\rm a},{\rm b},{\rm c}\})$. Let $e$ be the number of edges in $G$ that are incident with one vertex in $V_0$ and one vertex in $V_1$ and note that $e \leq \min(v_0d,v_1(d+k-1))$. In particular, $e=0$ if any of $v_0$, $v_1$ or $d$ equal $0$.

Let $c$ be a real number such that $c>\frac{d}{n}$ and $c$ is close to $\frac{d}{n}$, let $G^*$ be the $c$-reduced excess or $c$-reduced leave of $\mathcal{D}$ and note that $\sum_{u \in V_0}\wt_{G^*}(u) = ce$ and $\sum_{u \in V_1}\wt_{G^*}(u) \leq v_1(d+k-1)-(1-c)e$. There is an $m$-independent set $S$ in $G^*$ such that $|S| \geq s_{c}$ where, for $t \in \mathbb{R}$,
\begin{equation}
s_t=
\left\{
  \begin{array}{ll}
    v_0f_{m}\left(\mfrac{te}{v_0}\right)+v_1f_{m}\left(d+k-1-\mfrac{(1-t)e}{v_1}\right), & \hbox{if $e \geq 1$;} \\[0.1cm]
    v_0+v_1f_{m}(d+k-1), & \hbox{if $e=0$.}
  \end{array}
\right.
\end{equation}
The $e=0$ case follows by applying Lemma \ref{CTVariant} directly and using the fact that $f_m$ is monotonically decreasing, and the case $e \geq 1$ follows by applying Lemma \ref{CTVariantCor}(b) with $S_0=V_0$ and $S_1=V_1$ and again using the fact that $f_m$ is monotonically decreasing (note that $v_0,v_1 \geq 1$ if $e \geq 1$). By Lemma \ref{IndepSetCase2}, $\mathcal{D}$ has at least $s_{c}$ blocks. So, because $f_m$ is continuous (see (F1)) and we can choose $c$ arbitrarily close to $\frac{d}{n}$, $\mathcal{D}$ has at least $s_{d/n}$ blocks.

Equivalently, $\mathcal{D}$ has at least $h(e)$ blocks where $h$ is the function from the real interval $[0,\min(v_0d,v_1(d+k-1))]$ to $\mathbb{R}$ defined by
\begin{equation}\label{funcDef}
h(x)=
\left\{
  \begin{array}{ll}
    v_0f_{m}\left(\mfrac{dx}{nv_0}\right)+v_1f_{m}\left(d+k-1-\mfrac{(n-d)x}{nv_1}\right), & \hbox{if $x>0$;} \\[0.1cm]
    v_0+v_1f_{m}(d+k-1), & \hbox{if $x=0$.}
  \end{array}
\right.
\end{equation}
Note that $h$ is well defined because its domain is $\{0\}$ in the case where $v_0=0$ or $v_1=0$ and because if $v_1 \neq 0$ then $x \leq v_1(d+k-1)$ implies $d+k-1-\frac{(n-d)x}{nv_1} \geq 0$. We complete the proof by showing that $h(e) \geq \alpha_i v_0+\beta_i v_1$ for each $i \in \{{\rm a},{\rm b},{\rm c}\}$.

Observe that $h$ is continuous because $f_m$ is continuous (see (F1)) and $f_m(0)=1$. Also note that $\frac{dx}{nv_0} < m$ because $d<n$, $x \leq v_0d$ and $d \leq m$. Let
$z=\frac{nv_1(d+k-m-1)}{n-d}$ and observe that, if $v_1 \neq 0$, $d+k-1-\frac{(n-d)x}{nv_1} \leq m$ is equivalent to $x \geq z$. Thus, by applying the definition of $f_m$ and simplifying we obtain
\begin{equation}\label{funcPiecewise}
h(x) = \left\{
  \begin{array}{ll}
    h_1(x)=v_0+v_1\left(1-\mfrac{d+k-1}{2m}\right) + \mfrac{(n-2d)x}{2mn} & \hbox{if $x \geq z$ and $x > 0$;} \\[0.4cm]
    h_2(x)=v_0+\mfrac{nv_1^2(m+1)}{2nv_1(d+k)-2(n-d)x}-\mfrac{dx}{2mn} & \hbox{if $0 < x < z$;}
\\[0.4cm]
    h_3(x)=v_0+v_1\left(\mfrac{m+1}{2(d+k)}\right) & \hbox{if $x = 0$.}
  \end{array}
\right.
\end{equation}
We consider $h_1$ as a function from $\mathbb{R}$ to $\mathbb{R}$ and $h_2$ as a function from the real interval $(-\infty,\frac{nv_1(d+k)}{n-d})$ to $\mathbb{R}$. Note that $h_2$ is continuous on this domain and that $z \leq \frac{nv_1(d+k)}{n-d}$. Differentiating with respect to $x$ we see that
\begin{align*}
  h'_1(x) &= \mfrac{n-2d}{2mn}; \mbox{ and} \\
  h'_2(x) &= \mfrac{nv_1^2(m+1)(n-d)}{2(nv_1(d+k)-(n-d)x)^2}-\mfrac{d}{2mn}.
\end{align*}
So $h_1$ is monotonically increasing if $d \leq \frac{n}{2}$ and is monotonically decreasing if $d \geq \frac{n}{2}$. Note that, if $v_1,d \neq 0$, then $h'_2(x)$ has exactly one root in the domain we specified, namely $y = nv_1\left(\frac{d+k}{n-d}-\sqrt{\frac{m(m+1)}{d(n-d)}}\right)$. So, if $v_1,d \neq 0$, $h_2(x)$ is monotonically decreasing on the interval $(-\infty,y]$ and monotonically increasing on the interval $[y,\frac{nv_1(d+k)}{n-d})$. Finally, observe that, when $v_1,d \neq 0$, $y< z$ if and only if $d < \frac{mn}{2m+1}$.

\noindent{\bf Case (a).} From \eqref{funcPiecewise}, we have that $h(0) \geq \alpha_{\rm a}v_0+\beta_{\rm a}v_1$. So we may assume that $e>0$ and hence that $v_0,v_1,d \geq 1$. We have
$$f_{m}\left(\mfrac{de}{nv_0}\right) = 1-\mfrac{de}{2mnv_0} \geq 1-\mfrac{d^2}{2mn}=\alpha_{\rm a}$$
where the first equality follows from applying the definition of $f_m$, noting that $\frac{de}{nv_0} < m$, and the inequality follows from the fact that $e \leq v_0d$. We also have
$$f_{m}\left(d+k-1-\mfrac{(n-d)e}{nv_1}\right) \geq f_m(d+k-1) = \mfrac{m+1}{2(d+k)}=\beta_{\rm a}$$
where the inequality follows from the fact that $f_m$ is monotonically decreasing (see (F1)), and the first equality follows by applying the definition of $f_m$, noting that $d+k-1 \geq m$. Thus, from \eqref{funcDef}, we have $h(e) \geq \alpha_{\rm a}v_0+\beta_{\rm a}v_1$ as required.

\noindent{\bf Case (b).} Suppose that $d \geq \frac{n}{2}$ and that $d(d+k-1)<mn$. Note that when $v_1=0$, we have from \eqref{funcPiecewise} that $h(0) = v_0 = \alpha_{\rm b}v_0+\beta_{\rm b}v_1$. So we may assume that $v_1 \geq 1$. Because $d \geq \frac{n}{2} > \frac{mn}{2m+1}$, we have from our previous discussion of $h_2$ that $z<y$ and hence that $h_2$ is monotonically decreasing on the interval $[0,z]$. Because $d \geq \frac{n}{2}$, we have that $h_1$ is monotonically decreasing. Furthermore, $d(d+k-1)<mn$ implies that $v_1(d+k-1)>z$ because $v_1(d+k-1)-\frac{nv_1(d+k-m-1)}{n-d} = \frac{v_1}{n-d}(mn-d(d+k-1))>0$. Thus, from \eqref{funcPiecewise}, it follows that $h(e) \geq h_1(v_1(d+k-1))$ because $e \leq v_1(d+k-1)$. Applying the definition of $h_1$ and simplifying, we have
$$h(e) \geq h_1(v_1(d+k-1)) = v_0+v_1\left(1-\mfrac{d(d+k-1)}{mn}\right)=\alpha_{\rm b}v_0+\beta_{\rm b}v_1.$$

\noindent{\bf Case (c).} Suppose that $d < \frac{n}{2}$. Note that when $v_1=0$ or $d=0$, we have from \eqref{funcPiecewise}, that $h(0) = v_0 = \alpha_{\rm b}v_0+\beta_{\rm b}v_1$. So we may assume that $v_1,d \geq 1$. Because  $d < \frac{n}{2}$, $h_1$ is monotonically increasing. Thus, the global minimum of $h$ is at least the minimum of $h_2$ on the interval $[0,z]$, and it follows from \eqref{funcPiecewise} and our previous discussion of $h_2$ that this minimum is at least $h_2(y)$. Thus we have
$$h(e) \geq h_2(y)=v_0+v_1\left(\sqrt{\mfrac{d(m+1)}{m(n-d)}}-\mfrac{d(d+k)}{2m(n-d)}\right)=\alpha_{\rm c}v_0+\beta_{\rm c}v_1.$$
\end{proof}

Note that in the special case where $d=0$, the bound of Theorem \ref{d<rTrickyCoveringTheorem}(a) will usually be the strongest of our bounds. We now give examples of infinite families of parameter sets for which the bounds given by Theorem \ref{d<rTrickyCoveringTheorem} yield arbitrarily large improvements on the bound of Theorem \ref{NeatCoveringTheorem}. Let $\lambda$ be constant and $k \rightarrow \infty$. For $i \in \{{\rm a},{\rm b},{\rm c}\}$, let $\alpha_i$ and $\beta_i$ be as defined in the proof of Theorems \ref{d<rTrickyCoveringTheorem} and \ref{d<rTrickyPackingTheorem} and observe that
$$CB_{(v,k,\lambda)}(\alpha_i,\beta_i)-CB_{(v,k,\lambda)}(1,0)
= \mfrac{v(k\beta_i+r(\alpha_i-\beta_i-1)+\alpha_i-1)}{(k(\alpha_i-\beta_i)+1)(k+1)},$$
noting that $\beta_i>0$. This last expression is
\begin{equation}
\Theta(\tfrac{r}{k})(k\beta_i+r(\alpha_i-\beta_i-1)) - O(1)
\end{equation}
using the facts that $v = \Theta(kr)$ and that, for each $i \in \{{\rm a},{\rm b},{\rm c}\}$, $\alpha_i\leq 1$ and $\alpha_i-\beta_i = \Theta(1)$ because $\alpha_i > \frac{1}{2}$ and $\alpha_i \geq 2\beta_i$. Let $\Delta=k\beta_i+r(\alpha_i-\beta_i-1)$. If $\frac{r\Delta}{k} \rightarrow \infty$, then $CB_{(v,k,\lambda)}(\alpha_i,\beta_i)$ will become arbitrarily larger than the bound of Theorem \ref{NeatCoveringTheorem}.
\begin{itemize}
    \item
When $(\alpha_i,\beta_i)=(\alpha_{\rm a},\beta_{\rm a})$,
$$\Delta =\mfrac{(k-r)n(n+1)(n+2)-rd^2(d+k)}{\Theta(r^2k)}.$$
Thus, when $k-r = \Theta(k)$, $k = o(r^2)$ and $d = o(r)$ the bound of Theorem \ref{d<rTrickyCoveringTheorem}(a) yields arbitrarily large improvements on the bound of Theorem \ref{NeatCoveringTheorem}.
    \item
When $(\alpha_i,\beta_i)=(\alpha_{\rm b},\beta_{\rm b})$,
$$\Delta=(k-r)\left(1-\mfrac{d(d+k-1)}{n(n+1)}\right).$$
So when $r=\Theta(k)$, $k-r = \Theta(k)$ and $\frac{n}{2} \leq d \leq \frac{1}{2}(\sqrt{(k-1)^2+(4-\epsilon)n^2}-(k-1))$ for some positive constant $\epsilon$, we have $1-\frac{d(d+k-1)}{n(n+1)} = \Theta(1)$ and hence Theorem \ref{d<rTrickyCoveringTheorem}(b) yields arbitrarily large improvements on the bound of Theorem \ref{NeatCoveringTheorem}.
    \item
When $(\alpha_i,\beta_i)=(\alpha_{\rm c},\beta_{\rm c})$,
$$\Delta= \mfrac{(k-r)\sqrt{d}}{\Theta(r^2)}\sqrt{4(n+1)(n+2)(n-d)-d(d+k)^2}.$$
Observe that when $d<\min(\frac{n}{2},2n-k)$, we have
$$4(n+1)(n+2)(n-d)-d(d+k)^2 > \tfrac{n}{2}(4(n+1)(n+2)-(d+k)^2) = \Theta(r^2)$$
where the inequality follows because $d<\frac{n}{2}$ implies $(n-d)>\frac{n}{2}$, and the equality follows from $d < 2n-k$. So when $k-r = \Theta(k)$, $d<\min(\frac{n}{2},2n-k)$ and  $d \rightarrow \infty$, Theorem \ref{d<rTrickyCoveringTheorem}(c) yields arbitrarily large improvements on the bound of Theorem \ref{NeatCoveringTheorem}.
\end{itemize}
Table \ref{d<rTrickyCoveringImprovementsTable} gives examples of parameter sets for which one of the bounds of Theorem \ref{d<rTrickyCoveringTheorem} strictly improves on Theorem \ref{NeatCoveringTheorem}.

\begin{table}[H]
\begin{small}
\begin{center}
\begin{tabular}{|c|p{15cm}|}
\hline
$k$ & $v$ \\ \hline
$6$ & $21$ \\
$7$ & $24$, $25$, $30^{\rm c}$ \\
$8$ & $27$, $36$, $41$, $42$, $43$, $48^{\rm c}$ \\
$9$ & $32$, $33$, $40$, $41$, $47$, $48$, $49$, $55$, $56$, $57$, $62^{\rm c}$, $63^{\rm c}$, $64$, $65$ \\
$10$ & $35$, $37$, $44$, $45$, $53$, $54$, $55$, $63$, $64$, $72$, $73$, $77^{\rm b}$, $78^{\rm b}$, $79^{\rm c}$, $80^{\rm c}$, $81$, $82$ \\
$11$ & $40$, $50$, $51$, $60$, $61$, $69$, $70$, $71$, $78^{\rm c}$, $79^{\rm c}$, $80$, $81$, $90$, $91$, $96^{\rm b}$, $97^{\rm c}$, $98^{\rm c}$, $99^{\rm c}$, $100$, $101$ \\
$12$ & $44$, $54$, $55$, $56$, $65$, $66$, $67$, $76$, $77$, $78_{2}$, $86^{\rm c}$, $87^{\rm c}$, $88$, $89$, $96^{\rm b}$, $97^{\rm c}$, $98$, $99$, $100_{2}$, $107^{\rm c}$, $108^{\rm c}$, $111$, $117^{\rm b}$, $118^{\rm c}$, $119^{\rm c}$, $120$, $121$, $122$ \\
$13$ & $47$, $49$, $60$, $70$, $71$, $72$, $73$, $82$, $83$, $84$, $85$, $95^{\rm c}$, $96$, $97$, $105$, $106$, $107$, $108$, $109_{2}$, $117^{\rm c}$, $118$, $119$, $120$, $121$, $128^{\rm b}$, $129^{\rm c}$, $130^{\rm c}$, $132$, $133_{2}$, $140^{\rm c}$, $141^{\rm c}$, $142^{\rm c}$, $143$, $144$, $145$ \\ \hline

\end{tabular}
\end{center}
\end{small}
\vspace{-0.6cm}

\caption{For $\lambda=1$ and each $k\in \{3,\ldots,13\}$, the values of $v > \frac{13}{4}k$ for which one of the bounds of Theorem \ref{d<rTrickyCoveringTheorem} strictly improves on Theorem \ref{NeatCoveringTheorem}. Values of $v$ for which the bound of Theorem \ref{NeatCoveringTheorem} is improved by $i \geq 2$ are marked with a subscript $i$ and values of $v$ for which the bound of Theorem \ref{d<rTrickyCoveringTheorem}(b) or Theorem \ref{d<rTrickyCoveringTheorem}(c) is strictly greater than the bound of Theorem \ref{d<rTrickyCoveringTheorem}(a) are marked with a superscript b or c respectively.}
\label{d<rTrickyCoveringImprovementsTable}
\end{table}

\section{Some exact covering numbers}

This paper has focussed on establishing new lower bounds on covering numbers, but we conclude by
showing that for some parameter sets our new bounds are tight and yield exact covering numbers. For our purposes, an \emph{affine plane of order $q$} is a $(q^2,q,1)$-design. It is well known that an affine plane of order $q$  exists whenever $q$ is a prime power. The following result is based on a simple method for obtaining coverings from affine planes used in \cite{TodFP}.

\begin{lemma}\label{AffCovLemma}
Let $q$ be an integer such that an affine plane of order $q$ exists and let $s$ be a positive
integer, then $C(sq^2,sq) \leq q^2+q$.
\end{lemma}

\begin{proof}
Let $(U,\mathcal{A})$ be an affine plane of order $q$. Obviously $\mathcal{A}$ has $q^2+q$ blocks. Let $V = U \times [s]$ and let $\mathcal{B}=\{A \times [s]: A \in \mathcal{A}\}$. Then $(V,\mathcal{B})$ is an $(sq^2,sq,1)$-covering with $q^2+q$ blocks.
\end{proof}

Let $q$ be the order of an affine plane and let $s$ be a positive integer. By Lemma \ref{AffCovLemma}, $C(v,sq) \leq q^2+q$ for any $v \leq sq^2$. For parameter sets $(v,sq,1)$ where $sq^2-q+2 \leq v \leq sq^2$ and $s \geq q-1$, the Sch{\"o}nheim bound is $q^2+q$ and so we have $C(v,sq) = q^2+q$.  For parameter sets $(v,sq,1)$ where $v \leq sq^2-q+1$, however, the Sch{\"o}nheim bound is at most $q^2$. In the following result we show that Theorem \ref{NeatCoveringTheorem} and Theorem \ref{d>rCoveringTheorem}(a) allow us to conclude that $C(v,sq) = q^2+q$ for some parameter sets $(v,sq,1)$ where $v \leq sq^2-q+1$.

\begin{theorem}\label{ExactCN}
Let $q \geq 2$ be an integer such that an affine plane of order $q$ exists and let $s$ be an
integer such that $s \geq 2q+1$. Then $C(v,sq) = q^2+q$ for each integer $v$ such that
$sq^2-q+1-z < v \leq sq^2$, where
$$z=\left\{
      \begin{array}{ll}
        \min\left(q-1,\mfrac{q(s-2q-1)+2}{q+1}\right), & \hbox{if $2q+1 \leq s \leq 4q+1$;} \\[0.3cm]
        \mfrac{q^2(s-q-2)-q+2}{3q^2+3q-2}, & \hbox{if $s \geq 4q+2$.}
      \end{array}
    \right.
$$
\end{theorem}

\begin{proof}
Let $y$ be the largest integer less than $z$ and let $v'=sq^2-q+1-y$. It suffices to show that $C(v',sq) \geq q^2+q$, because then, for each integer $v$ such that $v' \leq v \leq sq^2$, we have
$$q^2+q \leq C(v',sq) \leq C(v,sq) \leq C(sq^2,sq) \leq q^2+q,$$
where the final inequality follows from Lemma \ref{AffCovLemma}.


\noindent{\bf Case 1.} Suppose that $2q+1 \leq s \leq 4q+1$. Then $z=\min(q-1,\frac{q(s-2q-1)+2}{q+1})$ and it follows from $s \geq 2q+1$ that $z>0$. Observe that $v'-1=q(sq-1)-y$ and $0 \leq y < q-1$. So, applying Theorem \ref{NeatCoveringTheorem} with $r=q$ and $d=y$, we have that $C(v',sq) \geq \lceil CB_{(v',sq,1)}(1,0)\rceil$ where
$$CB_{(v',sq,1)}(1,0)=\mfrac{(q+1)(sq^2-q+1-y)}{sq+1}.$$
Routine calculation shows that $CB_{(v',sq,1)}(1,0) > q^2+q-1$ if and only if
$y<\frac{q(s-2q-1)+2}{q+1}$.
Thus, by the definition of $v'$, we have $C(v',sq) \geq q^2+q$.

\noindent{\bf Case 2.} Suppose that $s \geq 4q+2$. Then $z=\frac{q^2(s-q-2)-q+2}{3q^2+3q-2}$ and it follows from $s \geq 4q+2$ that $z>q-1$. Observe that $v'-1=q(sq-1)-y$ and
$q-1 \leq y < sq-1$ (note that $y<z$ and it is easy to verify that $z \leq sq-1$). So, applying Theorem \ref{d>rCoveringTheorem} with $r=q$ and $d=y$, we have that $C(v',sq) \geq \lceil CB_{(v',sq,1)}(\tfrac{q}{2y+2},\tfrac{q}{2(y+sq)})\rceil$. Thus, by the discussion following Theorems \ref{d>rCoveringTheorem} and \ref{d>rPackingTheorem},
$$C(v',sq) \geq CB_{(v',sq,1)}(\tfrac{q}{2y+2},0)=\mfrac{q(q+1)(sq^2-q+1-y)}{sq^2+2y+2}.$$
Routine calculation shows that $CB_{(v',sq,1)}(\tfrac{q}{2y+2},0) > q^2+q-1$ if and only if $y<\frac{q^2(s-q-2)-q+2}{3q^2+3q-2}$. Thus, by the definition of $v'$, we have $C(v',sq) \geq q^2+q$.
\end{proof}

The results given by Theorem \ref{ExactCN} when $q=2$ or $q=3$ are already established, because exact covering numbers are known for parameter sets $(v,k,1)$ with $v \leq \frac{13}{4}k$ \cite{GreLiVan,Mil}. In Theorem \ref{ExactCN}, $z = \Theta(s)$ as $q \rightarrow \infty$ if $s \geq 4q+2$. This constitutes an improvement on a result of Todorov (see \cite[Corollary 4.5]{TodLB}). At the expense of more complication, a stronger result could be obtained by not making the simplification $C(v',sq) \geq  CB_{(v',sq,1)}(\tfrac{q}{2y+2},0)$ and by also employing Theorem \ref{d<rTrickyCoveringTheorem}(a). Table \ref{ExactCNTable} gives examples of parameter sets for which Theorem \ref{ExactCN} establishes exact covering numbers.

\begin{table}[H]
\begin{small}
\begin{center}
\begin{tabular}{|c|c|c||c|c|c||c|c|c||c|c|c|}
\hline
$k$ & $q$ & $v$ & $k$ & $q$ & $v$ & $k$ & $q$ & $v$ & $k$ & $q$ & $v$ \\
  \hline
$36$ & $4$ & $141$ &               $75$ & $5$ & $368,\ldots,371$ &         $105$ & $5$ & $518,\ldots,521$ &     $128$ & $4$ & $502,\ldots,509$ \\
$40$ & $4$ & $156, 157$ &          $76$ & $4$ & $298,\ldots,301$ &               & $7$ & $729$ &                $130$ & $5$ & $641,\ldots,646$ \\
$44$ & $4$ & $172, 173$ &          $80$ & $4$ & $314,\ldots,317$ &         $108$ & $4$ & $424,\ldots,429$ &     $132$ & $4$ & $518,\ldots,525$ \\
$48$ & $4$ & $187,\ldots,189$ &         & $5$ & $393,\ldots,396$ &         $110$ & $5$ & $542,\ldots,546$ &     $133$ & $7$ & $922,\ldots,925$ \\
$52$ & $4$ & $203,\ldots,205$ &    $84$ & $4$ & $329,\ldots,333$ &         $112$ & $4$ & $439,\ldots,445$ &     $135$ & $5$ & $666,\ldots,671$ \\
$55$ & $5$ & $271$ &               $85$ & $5$ & $418,\ldots,421$ &               & $7$ & $777, 778$ &           $136$ & $4$ & $534,\ldots,541$ \\
$56$ & $4$ & $219,\ldots,221$ &    $88$ & $4$ & $345,\ldots,349$ &         $115$ & $5$ & $567,\ldots,571$ &           & $8$ & $1081$ \\
$60$ & $4$ & $235,\ldots,237$ &    $90$ & $5$ & $443,\ldots,446$ &         $116$ & $4$ & $455,\ldots,461$ &     $140$ & $4$ & $550,\ldots,557$ \\
     & $5$ & $295, 296$ &          $92$ & $4$ & $361,\ldots,365$ &         $119$ & $7$ & $826, 827$ &                 & $5$ & $691,\ldots,696$ \\
$64$ & $4$ & $251,\ldots,253$ &    $95$ & $5$ & $468,\ldots,471$ &         $120$ & $4$ & $471,\ldots,477$ &           & $7$ & $970,\ldots,974$ \\
$65$ & $5$ & $320, 321$ &          $96$ & $4$ & $377,\ldots,381$ &               & $5$ & $592,\ldots,596$ &     $144$ & $4$ & $565,\ldots,573$ \\
$68$ & $4$ & $267,\ldots,269$ &    $100$ & $4$ & $392,\ldots,397$ &        $124$ & $4$ & $487,\ldots,493$ &           & $8$ & $1144, 1145$ \\
$70$ & $5$ & $344,\ldots,346$ &          & $5$ & $493,\ldots,496$ &        $125$ & $5$ & $616,\ldots,621$ &     $145$ & $5$ & $715,\ldots,721$ \\
$72$ & $4$ & $282,\ldots,285$ &    $104$ & $4$ & $408,\ldots,413$ &        $126$ & $7$ & $874,\ldots,876$ &     $147$ & $7$ & $1018,\ldots,1023$ \\ \hline
\end{tabular}
\end{center}
\end{small}
\vspace{-0.6cm}

\caption{For each $k\in \{3,\ldots,147\}$ and each choice of $q \geq 4$, the values of $v$ for which Theorem \ref{ExactCN} establishes that $C(v,k,1)=q^2+q$, excluding those for which this is implied by the Sch{\"o}nheim bound.}
\label{ExactCNTable}
\end{table}

\section{Conclusion}

It is worth noting that, via Lemma \ref{MainLemma}, improved bounds on the size of $m$-independent sets in multigraphs immediately translate to improved bounds on packing and covering numbers. The techniques employed in this paper may also produce useful results when applied to coverings and packings with blocks of various sizes. As mentioned in the discussion following it, Lemma~\ref{DominanceCor} need not require strict inequality in every row of the matrix. This raises the possibility of obtaining stronger results on coverings and packings in the special case where $d=r-\lambda$. There is also the potential to find further examples of coverings and packings meeting the new bounds and hence to exactly determine more covering and packing numbers. More speculatively, there is the possibility of attempting to obtain similar results for $t$-$(v,k,\lambda)$-coverings and packings for $t \geq 3$.

\vspace{0.3cm} \noindent{\bf Acknowledgements}

The author was supported by Australian Research Council grants DE120100040, DP120103067 and DP150100506.

\end{document}